\documentclass[reqno]{article}
\usepackage{amsmath}%
\usepackage{amsthm}
\usepackage{amsfonts}%

\usepackage{amssymb}%

\allowdisplaybreaks

\newtheorem{theorem}{Theorem}[section]

\newtheorem{corollary}[theorem]{Corollary}
\newtheorem{lemma}[theorem]{Lemma}
\newtheorem{proposition}[theorem]{Proposition}

\numberwithin{equation}{section}
\newtheorem{acknowledgements}{Acknowledgements}

\newcommand{\G}{{\mathcal G}}
\newcommand{\Lam}{{\mathcal L}}

\newcommand{\diag}{\mathop{\mathrm{diag}}\nolimits}

\providecommand{\keywords}[1]{\textbf{\textit{Keywords---}} #1}
\begin{document}

\title{Optimal BIBD-extended designs }
\author{ Cakiroglu, S. Aylin, \footnote{Present address: Francis Crick Institute,  1 Midland Road, NW1 1AT London, UK}\\
School of Mathematical Sciences, Queen Mary University of London,\\ Mile End Road, London E1 4NS, UK\\ \emph{e-mail: aylin.cakiroglu@crick.ac.uk} 
\and
 Cameron, Peter J.\footnote{Present address: School of Mathematics and Statistics 
University of St Andrews, North Haugh, St Andrews, Fife KY16 9SS, SCOTLAND}\\ 
School of Mathematical Sciences, Queen Mary University of London,\\ Mile End Road, London E1 4NS, UK\\ \emph{e-mail: pjc20@st-andrews.ac.uk}}

\maketitle

\begin{abstract}
Balanced incomplete block designs (BIBDs) are a class of designs with $v$ treatments and $b$ blocks of size $k$ that are optimal with regards to a wide range of optimality criteria, but it is not clear which designs to choose for combinations of $v$, $b$ and $k$ when BIBDs do not exist. 

In 1992, Cheng showed that for sufficiently large $b$, the designs which are optimal with respect to commonly used criteria (including the $A$- and $D$- criteria) must be found among $(M.S)$-optimal designs. 
In particular, this result confirmed the conjecture of John and Mitchell in 1977 on the optimality of regular graph designs (RGDs) in the case of large numbers of blocks.

We investigate the effect of extending known optimal binary designs by repeatedly adding  the blocks of a BIBD and find boundaries for the number of block so that these BIBD-extended designs are optimal. In particular, we will study the designs for $k=2$ and $b=v-1$ and $b=v$: in these cases the $A$- and $D$-optimal designs are not the same  but we show that this changes after adding blocks of a BIBD and the same design becomes $A$- and $D$-optimal amongst the collection of extended designs.
Finally, we characterise those RGDs that give rise to $A$- and $D$-optimal extended designs and extend a result  on the $D$-optimality of the a group-divisible design to $A$- and $D$-optimality amongst BIBD-extended designs.

\keywords{$A$-optimality, $D$-optimality, Incomplete block design, Regular graphs, Regular graph design, BIBD-extended design }

\end{abstract}

\section{Introduction and Preliminaries}\label{intro}

One of the main questions in the theory of optimal designs  is how to design experiments for the comparison of treatments when the available experimental units are  also affected by other nuisance or blocking factors.  The blocking factors are usually of no experimental interest, however they influence the measurements of the  treatment responses and must be accounted for in the experimental design. Specifically, suppose there are $v$ treatments to be compared on a number of experimental units that can be partitioned into $b$ blocks of size $k$ with $k<v$. These blocks might differ systematically but all units in a block are assumed to be alike. To each unit, one of the treatments will be applied, after which the response of the unit is measured. 
This assignment of treatments to the blocked experimental units is called a \emph{block design}. 
Since we are interested in the comparison of treatments, rather than estimating the response of a single treatment, we estimate a set of treatment contrast, that is a linear combination of the treatment effects whose coefficients sum up to zero. 
The variances of estimators of the unknown parameters are a function only of the treatment-unit assignment and can be calculated (up to a scalar multiple) before any measurements are taken. 
Therefore, the design optimality question often is the search for a treatment-unit assignment so that the estimate has the least possible variance.
This can be a multidimensional problem and a design can be deemed good in different ways. 

We give  the basic  definitions in the following and the reader is referred to \cite{Shah} for numerous references and  \cite{CameronBailey} where more details on the application of combinatorics in the theory of design of experiments can be found.
We will assume that both treatment and blocking factors affect the mean response additively, and that responses are otherwise subject to equivariable noise that is uncorrelated from one unit to another. 
If all pairwise differences  are estimable, the design is said to be \emph{connected}. 
The design is called \emph{binary}, if every treatment occurs at most once per block. We will always assume the designs to be binary and that $k<v$ (but not necessarily connected).
The \emph{replication} $r_i$ of a treatment $i$ is the total number of units that have been assigned treatment $i$, and  if the replications are all equal to a constant $r$, the design is called \emph{equireplicate}. 
Let $N_d$ be the $v\times b$ treatment-block incidence matrix of the design $d$, i.e. the $ij$-entry of 
$N_d$ is the number of units in block $j$ that have been assigned treatment $i$. The  \emph{concurrence} of treatments $i$ and $j$ is the  $ij$-entry of the product 
$N_dN_d^T$ and will be denoted by $\lambda_{ij}$. For binary designs, $\lambda_{ii}=r_i$ and $\lambda_{ij}$ is the number of blocks containing both treatments $i$ and $j$. 
The \emph{information matrix} for the estimation of the treatment effects is
\[
C_d=\diag(r_1,\ldots,r_v)-\frac{1}{k}N_dN_d^T.
\]
 Since $C_d$ has row sums $0$, the all-$1$-vector is an eigenvector with eigenvalue $0$. All other eigenvalues,  the \emph{non-trivial  eigenvalues}, are positive. If the design is connected, the rank of the information matrix is $v-1$, the eigenvalue $0$ has multiplicity $1$ and all non-trivial eigenvalues are strictly positive.  
The \emph{Laplacian matrix} of a binary connected  design is
\[
L_d=k \diag(r_1,\ldots,r_v)-N_dN_d^T=kC_d.
\]
The smallest eigenvalue of $L_d$ is $0$ and the other $v-1$ \emph{non-trivial Laplacian eigenvalues}  are all positive (and strictly positive if the design is connected). 
The non-zero eigenvalues of $C_d$ are the reciprocals of the variances  for a  set of  orthonormal treatment contrasts (aside from the variance of the noise). Therefore, we can define measures how good a treatment-unit assignment (or design) is based on summary functions of the non-trivial eigenvalues: 
The average variance  of the set of the best linear unbiased estimators of the pairwise differences of the treatment effects is proportional to the reciprocal of the harmonic mean of the non-trivial eigenvalues of $C_d$, and a design $d$ that minimizes this summary function is said to be $A$-optimal.  The volume of the confidence ellipsoid for any orthonormal contrasts is proportional to the product of the reciprocals  of the non-zero eigenvalues of $C_d$, and a design $d$ that minimizes this product is said to be $D$-optimal.  
There are many more optimality criteria; another popular example is the $E$-criterion which is the maximization of the smallest non-trivial eigenvalue of $C_d$ and is equivalent with minimizing the largest variance of the estimators.
A design that maximises the sum of the non-trivial eigenvalues and minimizes the sum of squares of the entries of the information matrix (among those that maximise the sum of the non-trivial eigenvalues) is called \emph{(M.S)-optimal} \cite{EcclestonHedayat}. However,  the $(M.S)$-optimal designs are not unique in their class and often not efficient on the other optimality criteria. If they exist, RGDs are $(M.S)$-optimal \cite{ChengMS}.
 A design is called \emph{Schur-optimal} if the non-trivial eigenvalues of $C_d$ majorize the non-trivial  eigenvalues of the information matrix of any competing design. Schur-optimality is a most general optimality criterion: a Schur-optimal design  minimizes any \emph{Schur-convex} function  of the non-trivial Laplacian eigenvalues, that is any function such that,  if $x\in {\mathcal A}\subset  {\mathbb R}^n$ is majorized by $y\in {\mathcal A}$ then $\Phi(x)\geq\Phi(y)$, whenever $x$ is not a permutation of $y$; examples of Schur-convex functions are the $A$-, $D$- and $E$- criteria.
Discussions of these and other optimality criteria can be found in \cite{Shah}.

There is no general answer to the question of which design is to be chosen for given $v$, $b$ and $k$, but there are several partial results for certain choices of $v$, $b$ and $k$:  \emph{balanced incomplete block designs} (in short BIBDs) are binary equireplicate incomplete block designs with replication $bk/v$, where $k<v$ and any pair of treatments is contained in exactly $\lambda$ blocks for some $\lambda>0$. BIBDs are optimal with regards to a wide range of criteria, in particular the $A$- and $D$-criteria \cite{Kiefer}, but it is not clear which designs to choose if no BIBD exists. 
Other popular designs are the \emph{regular graph designs} (in short RGDs); these are equireplicate binary designs in which any pair of points occurs in either $\lambda$ or $\lambda+1$ blocks for some integer $\lambda\geq 0$.The Laplacian matrix of an RGD $d$ with $v$ treatments, replication $r$ and block size $k$ can  be written as
\[
L_d=\{r(k-1)+\lambda\}{I}_v-{T}_d-\lambda{J}_v,
\]where  ${I}_v$ is the $v\times v$ identity matrix, ${J}_v$ denotes the $v\times v$ all-$1$-matrix and ${T}_d$ is a symmetric $v\times v$ $(0,1)$-matrix with $0$'s on the diagonal and exactly $r(k-1)-\lambda(v-1) $ number of $1$'s in each row and each column. Note that the Laplacian matrix of an RGD is therefore fully determined by  the parameters $v$, $k$, $r$ and the matrix ${T}_d$. 

The reference to graphs in the names of the RGDs is hinting towards a close relationship between  block designs and graphs that hinges on the  notion of the Laplacian matrix. First, we will need some basic definitions:
a graph $\G$ is a set of $v$ vertices and a set of edges that connect vertices $i \not= j$ (we do not allow any loops). The \emph{adjacency matrix} $A_{\G}$ of the graph $\G$ is the $v\times v$ matrix whose $ij$-entry is the number of edges joining vertices $i$ and $j$. A graph $\G$ is \emph{connected} if any vertex can be reached from any other vertex by going along edges.  We say that $\G$ is \emph{simple}, if $\G$ contains no multiple edges. The \emph{degree $\delta_i$ of the vertex} $i$ is the number of edges incident to $i$. If the degrees of all its vertices are equal to a constant $\delta$, the \emph{degree} of the graph, the graph is called \emph{regular} (note that its adjacency matrix has constant row and column sums).

The Laplacian matrix $L_{\G}$ of a graph $\G$ on $v$ vertices is defined as 
\[ 
L_{\G} = \diag(\delta_1,\ldots,\delta_{v}) - A_{\G}.
\]  Note that $L_{\G}$ is
a symmetric matrix with row and column sums zero.  A binary block design can be represented as a graph, called the \emph{concurrence graph}, by taking the treatments as vertices and joining any two distinct vertices $i$ and $j$ are joined by $\lambda_{ij}$ edges, where
$\lambda_{ij}$ is the concurrence of $i$ and $j$. Note that the concurrence
graph contains no loops, even if the design is not binary, but is not necessarily simple. If $\G$ is the concurrence graph of a binary design $d$ with block size $k$, then $\delta_i = r_i(k-1)$, $i=1,\ldots,v$ and then $L_{\G}= L_{d}$. 

All regular simple graphs with $v$ vertices of degree $\delta$  correspond to all symmetric $v\times v$-matrices with $(0,1)$-entries, zero diagonal and row and column sum $\delta$; that is the  matrix $T_d$ of a RGD $d$ is precisely the adjacency matrix of a simple regular graph of degree $r(k-1)-\lambda(v-1)$. 
John and Mitchell  conjectured that
if an incomplete block design is $D$-optimal (or $A$-optimal or $E$-optimal), then it is an RGD (if any RGDs exist) \cite{JohnMitchell}. This conjecture has been shown to be wrong for all three optimality criteria \cite{JonesEccleston, Constantine, RosemaryExamples}, but holds if the number of blocks is large enough \cite{ChengMS}.

In this context, we want explore the boundaries of when RGDs become optimal and the landscape of optimal designs with a large number of blocks. In patricular, we are interested  in the effect of adding the blocks of a BIBD  repeatedly to a design  on the performance of the design on the $A$- and $D$- criteria (for a treatment of the $E$-criterion see \cite{Morgan}). In the spirit of \cite{Morgan}, we call these designs \emph{BIBD-extended} designs. 
We will define a partial order on the Laplacian matrices and identify some bounds on the number of blocks for which BIBD-extended RGDs are $A$- and $D$-optimal among all competing BIBD-extended binary designs in Section \ref{order}. We will use these bounds in Section \ref{example} to prove the $A$- and $D$-optimality of some of the BIBD-extended designs. In Section \ref{RGD}, we will identify some bounds on the number of blocks for which a class of BIBD-extended RGDs are $A$- and $D$-optimal among BIBD-extended RGDs and use these to prove the $A$- and $D$-optimality of certain BIBD-extended group-divisible designs among BIBD-extended RGDs.

\begin{acknowledgements}
SAC would like to thank R. A. Bailey for the many conversations on optimal designs and is extremely grateful to  J.\, P. Morgan for bringing the topic to her attention and the helpful discussions about RGDs.
\end{acknowledgements}

\section{A partial order on BIBD extended designs and $A$- and $D$-optimality}\label{order}

Suppose, the binary design $d$ has Laplacian matrix $L_d$ and $\tilde{d}$ is a BIBD on $v$ treatments and block size $k$ with  $\tilde{b}$ blocks and Laplacian matrix $L_{\tilde{d}}$ and concurrence parameter $\tilde{\lambda}$.  Then for  $y\in {\mathbb N}$  the matrix $L_d+yL_{\tilde{d}}$
is the Laplacian matrix of a \emph{BIBD-extended}  design on $v$ points,  replication $r+y\tilde{\lambda}(v-1)/(k-1)$ and $b+y\tilde{b}$ blocks of size $k$.
Let $\rho_1^{L_d},\ldots,\rho_{v-1}^{L_d}$ be the non-trivial Laplacian eigenvalues of $d$, then the BIBD-extended desgin has the non-trivial eigenvalues  $vy+\rho_1^{L_d},\ldots,vy+\rho_{v-1}^{L_d}$.
We can write the $A$- and $D$-value as functions that only depend on $y$ and $L_d$ as 

\begin{eqnarray*}A(y,L_d)&:=&\frac{v-1}{\sum_{i=1}^{v-1}\frac{1}{vy+\rho_i^{L_d}}} \text{ and }\\ 
D(y,L_d)&:= &\prod_{i=1}^{v-1}(vy+\rho_i^{L_d}),\text{ respectively. }
\end{eqnarray*}

For  $z = (z_1,\ldots, z_{v-1})\in{\mathbb R}^{v-1}$ and $ \mathcal{I}=\{1,\ldots, v-1\}$, let \begin{eqnarray*}
S_0(z)\equiv 1 \text{ and } S_j(z):=\sum_{\substack{\mathcal{J}\subseteq \mathcal{I}\\ |\mathcal{J}|=j}}\prod_{i\in \mathcal{J}} z_i. 
 \end{eqnarray*}
$S_j$ is called  \emph{the $j$th elementary symmetric polynomial}.
Then  for $\rho^{L_d}:=(\rho_1^{L_d},\ldots,\rho_{v-1}^{L_d})$ \cite{Cakiroglu}
\begin{eqnarray}\label{Dvalue general}
D(y,L_d) &=&  \sum_{j=0}^{v-1}(vy)^{v-1-j}S_j( \rho^{L_d}), \text{ and }\\ \label{Avalue general}
A(y,L_d) &= &\frac{(v-1)\sum_{j=0}^{v-1}(vy)^{v-1-j}S_j( \rho^{L_d})}{\sum_{j=0}^{v-2}(vy)^{v-2-j}(v-1-j) S_j( \rho^{L_d}) }.
\end{eqnarray}
Note that the $A$-value is a rational function whose coefficients are completely determined by the coefficients of the $D$-value and  $\rho_1^{L_d},\ldots,\rho_{v-1}^{L_d}$  are also the Laplacian eigenvalues of the concurrence graph, $\G_d$ of $d$.
For $j\in\{1,\ldots,v-1\}$, there exists a relationship between $S_{v-j}(\rho^{\G_d})$ and the set ${\mathcal F}_j$ of all disjoint unions $F_{j}$ of $j$ subgraphs of $\G$ on $n_1,\ldots,n_j$ vertices in which any two vertices are connected by exactly one path (for example \cite{SpectrOfGraphs}, p. 38):
\[
S_{v-j}=\sum_{
F_{j}\subset {\mathcal F}_{j}}\gamma(F_{j}),\] 
where  $\gamma(F_{j})=\prod_{k=1}^j n_k$ for all $F_j\in{\mathcal F}_j$. 
In particular, $S_j(\rho^{L_d})\in \mathbb{N}_{>0}$ for $j = 1,\ldots, v$.

Because the performance of a BIBD-extended binary design on the $A$- and $D$-criteria is in this way closely related to the elementary symmetric polynomials of its non-trivial Laplacian eigenvalues, we want to define an order on the matrices  accordingly. We will do this in a similar fashion to \cite{Constantine} and take the lexicographic ordering corresponding to the elementary symmetric polynomials of the non-trivial eigenvalues of the Laplacian matrices. That means, for two Laplacian matrices $L, L'\in  \Lam(v,b,k)$ whose set of eigenvalues do not  coincide, we will write $L'\prec L$ if   there  exists an $l\in\{1,\ldots,v-1\}$ such that
\[
 S_l(\rho^{L'})< S_l(\rho^{L})\text{ and }S_j(\rho^{L})= S_j(\rho^{L'}) \text{ for }\ j=1,\ldots,l-1.
\] This is a reflexive and transitive relation which we will call  the \emph{stable order} on $\Lam(v,b,k)$. Note that matrices are indistinguishable in this order if  they have the same eigenvalues. 
From Equation (\ref{Dvalue general}) it is immediately clear that if $L'\prec L$, then there exists a $y_0\geq 0 $ such that $D(y,L')\leq D(y,L)$ for $y\geq y_0$. If $L'\prec L$ such that $l$ is the smallest index with $S_l(\rho^{L})-S_l(\rho^{L'})\not = 0$, then the  polynomial
\begin{eqnarray*}
P(y,L,L') :=\sum_{i=1}^{2v-3}p_iy^i &=& \sum_{j=0}^{v-1}(vy)^{v-1-j}S_j( \rho_i^{L}) \sum_{j=0}^{v-2}(vy)^{v-2-j}(v-1-j) S_j( \rho_i^{L'})   \\
 &  &-   \sum_{j=0}^{v-1}(vy)^{v-1-j}S_j( \rho_i^{L'})\sum_{j=0}^{v-2}(vy)^{v-2-j}(v-1-j) S_j( \rho_i^{L})
\end{eqnarray*}
has coefficients 
\begin{eqnarray*}
p_{2v-3-i}&=&v^{2v-3-i}\sum_{j=0}^{i}(v-1-j)\left[S_{j}(\rho^{L'})S_{i-j}(\rho^{L})-S_{j}(\rho^L)S_{i-j}(\rho^{L'})\right]\\ 
\end{eqnarray*} that are equal to zero for  $i = 1,\ldots, l-1$,  
and the first non-vanishing coefficient is 
\begin{eqnarray*}
p_{2v-3-l}
&=&v^{2v-3-l}l\left[S_{l}(\rho^{L})-S_{l}(\rho^{L'})\right] \in {\mathbb N}_{>0}.
\end{eqnarray*} 
Since $P(y,L,L')\geq 0$ implies $A(y,L')\leq A(y,L)$ (Equation (\ref{Avalue general})), it follows that there exists a $y_0 \geq 0 $ such that $A(y,L')\leq A(y,L)$ for $y\geq y_0$.

\begin{corollary}\label{proposition stable order for y_0}
For given $v,b,k$, there exists a $y_0\geq 0$ such that, if designs with Laplacian matrix $L[y]$ exist for some $L\in  \Lam(v,b,k)$ and $y\geq y_0 $, then their order under the $A$- or $D$-criterion is the stable order of the matrices $L$. 
\end{corollary}

For $i=1,\ldots,v-1$, we will denote by $\Lam_i(v,b,k)$ the set of Laplacian matrices $L\in\Lam(v,b,k)$  such that \[S_j(\rho^{L})=\max\{S_j(\rho^{L'})|L'\in\Lam(v,b,k)\}\text{ for }j=1,\ldots,i.\]
That means, that for $i\in\{1,\ldots,v-2\}$ and for $L,L'\in\Lam_i(v,b,k)$ with $L'\not\in\Lam_{i+1}(v,b,k)$ we have
$L'\prec L$.
All binary designs conincide on $S_1$, and  hence  $\Lam_1(v,b,k)$ is the set of the Laplacian matrices of all binary designs \cite{Kiefer2}; $\Lam_2(v,b,k)$ is the set of the Laplacian matrices of $(M.S)$-optimal designs, these are any existing regular graph designs (RGDs) or  nearly balanced incomplete block designs (NBDs) \cite{ChengMS, chengWu}. 
We obtain a similar result to that  in \cite{Constantine}.
\begin{corollary}\label{Corollary Constantine}
 For given $v$, $b$, $k$ and $2\leq j\leq v$,there exists a $y_0\geq 0$ such that, if  designs   with Laplacian matrix $L[y]$ where $L\in  \Lam_{j}(v,b,k)$ exist  for some $y\geq y_0$, then they are $A$- and $D$-optimal among all designs with Laplacian matrices $L'[y]$ with $L'\in\Lam_1(v,b,k)\setminus \Lam_j(v,b,k)$.
\end{corollary}

\begin{lemma}\label{lemma bound largest eigenvalue}
Let $L=(L_{ij})\in \Lam_1(v,b,k)$ and $\rho_1^L\geq \ldots\geq \rho_{v-1}^L>0$, then $\rho_1^L\leq 2b(k-1)$.
\end{lemma}
\begin{proof}
$L$ is the Laplacian matrix of a connected graph with degrees $\delta_i=r_i(k-1)$ for $i = 1,\ldots,v$. 
The Laplacian eigenvalues of a connected graph are bounded from above (for example \cite{Das}) by $\max\{\delta_u+\delta_w|(u,w)\in E(\G)\}$. Since $r_i\leq b$, it follows $\rho_1^L\leq \max\{(r_i+r_j)(k-1)|L_{ij}\not=0, i\not=j\} \leq 2b(k-1)$.
\end{proof}
\begin{lemma}\label{D-value Bionom}
The binomial coefficient $2^j\binom{v-1}{j}$ (as a function in $j = 1,\ldots,v-1$) is monotone increasing for $j\leq \lfloor\frac{2v-3}{3}\rfloor$  and attains its  maximum at $\lfloor\frac{2v-3}{3}\rfloor$
\end{lemma}
\begin{proof}
Follows directly from 
\[
\frac{2\binom{v-1}{j+1}}{\binom{v-1}{j}}=\frac{2(v-1-j)}{j+1}>1 \text{ iff }j< \frac{2v-3}{3},\ j=3,\ldots,v-1.
\]
\end{proof}

\begin{lemma}\label{A-value Bionom}

The binomial coefficient $\binom{v-1}{j}$ as a function in $j$ is strictly increasing until $j<\frac{v-2}{2}$ and the maximum is attained at $\lfloor\frac{v-2}{2}\rfloor$. 
\end{lemma}
\begin{proof}
Follows directly from 
\[
\frac{\binom{v-1}{j+1}}{\binom{v-1}{j}}=\frac{v-1-j}{j+1}.
\] 
\end{proof}

\begin{proposition}\label{Special case of Cheng}
Suppose, $\Lam_1(v,b,k)\not=\Lam_2(v,b,k)$.
\begin{enumerate} 
\item For given $v,b,k$ and $y_0=v^22^m(b(k-1))^{v-1} \binom{v-1}{m}
+1$, where $m=\lfloor\frac{2v-3}{3}\rfloor$, if there exist  designs with Laplacian matrix $L[y]$ with $L\in\Lam_2(v,b,k)$ and $y\geq y_0$, then they are $D$-optimal among all designs with Laplacian matrix $L'[y]$ with $L'\in\Lam_1(v,b,k)$.

\item For given $v,b,k$ and $y_0=2^{v-2}(b(k-1))^{v-1}(2v-5) \binom{v-1}{m}^2+\tfrac{1}{v}$, where $m=\lfloor\frac{v-2}{2}\rfloor$, if there exist  designs with $L[y]$ with $L\in\Lam_2(v,b,k)$ and $y\geq y_0$, then they are $A$-optimal among all designs with Laplacian matrix $L'[y]$ with $L'\in\Lam_1(v,b,k)$.
\end{enumerate}
\end{proposition} 
\begin{proof}

\begin{enumerate}

\item
Let $m=\lfloor\frac{2v-3}{3}\rfloor$.
With Lemmas \ref{lemma bound largest eigenvalue} and \ref{D-value Bionom}
 \[
 |S_j(\rho^{L})- S_j(\rho^{L'})|\leq (b(k-1))^j2^{m}\binom{v-1}{m}.
 \]
Since $S_j(\rho^{L})-S_j(\rho^{L'}) \in {\mathbb N}$ for $j=1,\ldots,v-1$, it follows for $y\geq v^22^m(b(k-1))^{v-1} \binom{v-1}{m}+1$ that
\begin{eqnarray*}
 \sum_{j=3}^{v-1}v^{v-1-j}y^{v-1-j}|S_j(\rho^{L})- S_j(\rho^{L'})|
&<&2^m(vb(k-1))^{v-1}\binom{v-1}{m}\sum_{j=3}^{v-1}y^{v-1-j}\\
&\leq&2^m(vb(k-1))^{v-1}\binom{v-1}{m}\frac{y^{v-3}-1}{y-1}\\
&\leq&v^{v-3}(y-1)\frac{y^{v-3}-1}{y-1}\\
&<&v^{v-3}y^{v-3}\\
&\leq &v^{v-3}y^{v-3}[S_2(\rho^{L})-S_2(\rho^{L'})].
\end{eqnarray*}

\item  
Let $m=\lfloor\frac{v-2}{2}\rfloor$.
With Lemmas \ref{lemma bound largest eigenvalue} and \ref{A-value Bionom} it follows for all $i=1,\ldots,v-1$, that
\[
|S_j(\rho^{L}) S_{i-j}(\rho^{L'})-S_j(\rho^{L'}) S_{i-j}(\rho^{L})|\leq (2b(k-1))^i\binom{v-1}{m}^2,\ j=3,\ldots,v-1.
\]
For $y-\tfrac{1}{v}\geq 2^{v-2}(b(k-1))^{v-1}(2v-5) \binom{v-1}{m}^2$ 
\begin{eqnarray*}
 &&\sum_{i=3}^{v-1}(vy)^{2v-3-i}\sum_{j=0}^i(v-1-j)|S_j(\rho^{L}) S_{i-j}(\rho^{L'})-S_j(\rho^{L'}) S_{i-j}(\rho^{L})|\\&\leq&(2b(k-1))^{v-1}\binom{v-1}{m}^2\left[\sum_{i=3}^{v-1}(vy)^{2v-3-i}\sum_{j=0}^i(v-1-j)\right]\\
&=&(2b(k-1))^{v-1}\binom{v-1}{m}^2\left[\sum_{i=3}^{v-1}(vy)^{2v-3-i}(i+1)\frac{2(v-1)-i}{2}\right]\\
&<&(2b(k-1))^{v-1}\binom{v-1}{m}^2\frac{v(2v-5)}{2}(vy)^{2v-6-(v-4)}\sum_{i=0}^{v-4}(vy)^{i}\\
&\leq&(vy-1)(vy)^{v-2}\frac{(vy)^{v-3}-1}{vy-1}\\
&<&2(vy)^{2v-5}\\
&\leq & 2(vy)^{2v-5}\left[S_{2}(\rho^{L})-S_{2}(\rho^{L'})\right].
\end{eqnarray*}
\end{enumerate}
\end{proof}
We believe that the bounds in Proposition \ref{Special case of Cheng} almost certainly can be improved in general. In particular, we will show that for certain designs $y_0=1$;  we will need the following lemma. Before stating it, we would like to note that we found reading \cite{Stevanovic2} very helpful in developing its proof and we use similar ideas.

  \begin{lemma}\label{A-value increases} Let 
  ${\mathcal M} = \{ z = (z_1,\ldots,z_{v-1}) \in {\mathbb R}^{v-1}| z_1\geq z_2\geq \ldots \geq z_{v-1}>0\}$.
   For fixed $y>0$,  the function of $(S_1(z), \ldots,S_{v-1}(z))$ given by
   \[\mathcal{A}_y: \left\{ 
   \begin{array}{rl}
    \{(S_1(z),\ldots,S_{v-1}(z))|z \in {\mathcal M}\} \rightarrow &{\mathbb R} ,\\
    (S_1(z),\ldots,S_{v-1}(z)) \mapsto &\frac{(v-1)\sum_{j=0}^{v-1}(vy)^{v-1-j}S_j( z)}{\sum_{j=0}^{v-2}(vy)^{v-2-j}(v-1-j) S_j( z) }
\end{array}    \right. \]
   is increasing in $S_j(z)$,  $j\in\{1,\ldots,v-1\}$.
  \end{lemma}
  \begin{proof}

  Let  $i\in \mathcal{I}=\{1,\ldots,v-1\}$ and $z \in {\mathcal M}$. Then 
\begin{eqnarray*}
\frac{\delta S_j(z) }{ \delta  z_i }&=& \sum_{\substack{i \in \mathcal{J}\subseteq\mathcal{ I}\\ |\mathcal{J}|=j}}\prod_{k\in \mathcal{J}\setminus \{i\}} z_k>0 
\end{eqnarray*}and 
\begin{eqnarray*}
 \frac{\delta \mathcal{A}_y(z)}{ \delta z_i} 
= \frac{v-1}{\left(vy+z_i\right)^2\left(\sum_{i=1}^{v-1}\frac{1}{vy+z_i}\right)^2} >0.  \nonumber 
\end{eqnarray*} By chain-rule, it follows that
\begin{eqnarray*}
\frac{\delta \mathcal{A}_y((S_1(z),\ldots,S_{v-1}(z)))}{\delta S_j(z)} &=&\sum_{i=1}^{v-1}\frac{\delta \mathcal{A}_y(z)}{\delta z_i} \frac{\delta z_i}{\delta S_j(z)} 
= \sum_{i=1}^{v-1}\frac{\delta \mathcal{A}_y(z) / \delta z_i} {\delta S_j(z) / \delta z_i}
> 0.
\end{eqnarray*}
  \end{proof}

\begin{theorem}\label{Theorem1}

 For given $v$, $b$, $k$, if  designs   with Laplacian matrix $L[y]$ where $L\in  \Lam_{v-1}(v,b,k)$ exist  for some $y>0$, then they are $A$- and $D$-optimal among all designs with Laplacian matrices $L'[y]$ with $L'\in\Lam_1(v,b,k)$.

\end{theorem}

\begin{proof}
For the $D$- value this is immediately clear and for the $A$-value it follows from Lemma \ref{A-value increases}.
\end{proof}
Note that for $j=1,\ldots,v-1$, the function $S_j$ is increasing and Schur-concave on ${\mathbb R}^{v-1}_{\geq0}$, and if $j\not=1$, then $S_j$ is strictly Schur-concave on ${\mathbb R}^{v-1}_{>0}$ (for example \cite{Beckenbach}, pp. 78).
In particular, if a Schur-optimal design with Laplacian matrix $L\in  \Lam_1(v,b,k)$ exists, then $L\in  \Lam_{v-1}(v,b,k)$.

A special emphasis lies on the fact that the BIBD-extended designs in Theorem \ref{Theorem1} are both $A$- \emph{and} $D$-optimal.   It has been conjectured that among RGDs the same design is $A$-optimal and $D$-optimal \cite{JohnMitchellConjectures}.   Theorem \ref{Theorem1} proves the conjecture for a special case  of BIBD-extended RGDs, however the conjecture is not true in general \cite{Cakiroglu}. 

\section{$A$- and $D$-optimal BIBD-extended designs in $\Lam_{v-1}(v,b,2)$, $b\in \{v-1,v\}$ }\label{example}

For $k=2$, every block of a binary design is represented by an edge in the concurrence graph  and we will not make a distinction between the graph and the design. In this section we assume all designs to be connected.

\subsection{The case $b=v-1$}

Among connected  designs, the $A$-optimal design  is   the \emph{reference design}; this is a design that has one treatment occuring in every block and all other blocks contain each other treatment \cite{RosemaryExamples}. Its concurrence graph is the star graph where one vertex is connected to every other vertex by exactly one edge and no other edges. 
All connected graphs on $v$ vertices and $v-1$  perform equally on the $D$-criterion for $y=0$ \cite{RosemaryExamples}. The path $Path(v)$ is another graph in this class;  this graph has vertices $w_1,\ldots,w_v$ such that $w_i$ and $w_{i+1}$ are joined by an edge for $i=1,\ldots,v-1$. 
The Laplacian matrix of the path $Path(v)$  is in $\Lam_{v-1}(v,v-1,2)$ and the star graph minimizes $S_j$ for all $j=1,\ldots,v-1$ among all connected graphs on $v$ vertices with $v-1$ edges \cite{ZhouGutman}. 
This proves the following proposition. 
\begin{proposition}\label{propopsition path}
 Any design with Laplacian matrix $L(Path(v))[y]$ is  $D$-opti\-mal ($A$-optimal) for and $y\geq 0$ ($y\geq 1$) among  connected designs with Laplacian matrix $L[y]$ with $L\in\Lam_1(v,v-1,2)$.
\end{proposition}

\subsection{The case $b=v$}

The cycle $Cycle(v)$ on $v$ vertices is a path with vertices $w_1,\ldots,w_v,w_{v+1}=w_1$. Bailey showed, that  $Cycle(v)$ is the adjacency graph of a connected $D$-optimal  design and a connected $A$-optimal  design only for $v\leq 8$ and $v=12$  \cite{RosemaryExamples}. For $9\leq v\leq 11$, the quadrangle of which one vertex is joined to all the remaining $v-4$ vertices  is $A$-optimal. For $v\geq 13$ the triangle of which one vertex is joined to the remaining  $v-3$ vertices is $A$-optimal; we will denote this graph by $C_3(v-3)$. For $v=12$, both $Cycle(12)$ and $C_3(9)$ give $A$-optimal binary designs. It is  interesting to note  that   $C_3(v-3)$ minimizes $S_j$ for all $j=2,\ldots,v-1$ among all simple connected graphs with $v$ edges, and further $Cycle(v)$ maximizes $S_j$ for all $j=2,\ldots,v-1$ among all simple connected graphs with $v$ edges  \cite{Stevanovic}. Since all designs in the class with Laplacian matrix in $\Lam_2(v,v,2)$ have a simple  concurrence graph, it follows that $L(Cycle(v))\in\Lam_{v-1}(v,v,2)$. This proves the following proposition.

\begin{proposition}\label{propopsition cycle}
 Any design with Laplacian matrix $L(Cycle(v))[y]$ is  $D$-opti\-mal ($A$-optimal) for any $y\geq 0$  ($y\geq 1$)  among  connected designs with Laplacian matrix $L[y]$ with $L\in\Lam_1(v,v,2)$.
\end{proposition}

\section{$A$- and $D$-best BIBD-extended RGDs}\label{RGD}
From now on we assume that $bk/v=r\in{\mathbb N}$, i.e. $\Lam_2(v,b,k)$ is the set of Laplacian matrices of any existing RGDs.
Let $d$ be an RGD with Laplacian matrix $L_d\in \Lam_2(v,b,k)$ with $\lambda=\lfloor r(k-1)/(v-1)\rfloor$,  and $\tilde{d}$ be a BIBD on $v$ treatments and $\tilde{b}$ blocks of size $k$ with Laplacian matrix $L_{\tilde{d}}$ and concurrence parameter $\tilde{\lambda}$. Further, let $x=\lambda+y\tilde{\lambda}$. We can write the Laplacian matrix $L_d+yL_{\tilde{d}}$ of the BIBD-extended RGD as 
\begin{eqnarray*}
L[x,{T}_d]=(\delta+vx){I}_v-{T}_d-x{J}_v, \ \delta=r(k-1)-\lambda(v-1).
\end{eqnarray*}
The matrix $T_d$ is the adjacency matrix of a regular, simple graph $\G_d$. 
Let $\psi_1^{{\G}_d}\geq \psi_2^{{\G}_d}\geq\ldots \geq \psi_v^{{\G}_d}$ denote the eigenvalues of $L(\G_d) = \delta{I}_v-{T}_d$. Then,  $\psi_{v}^{{\G}_d}=0$ since ${\G}_d$ has row sum $\delta$. The non-trivial eigenvalues of $L[x,\G_d]$ are $vx+\psi_1^{{\G}_d},\ldots,vx+\psi_{v-1}^{{\G}_d}$. As before, we can express the $A$- and $D$-value as a rational function and polynomial in $y$, respectively: 
\begin{eqnarray} \label{A-value as D/D'}
A(x,\G_d)
&=&\frac{(v-1)\sum_{j=0}^{v-1}(vy)^{v-1-j}S_j( \psi_i^{{\G}_d})}{\sum_{j=0}^{v-2}(vy)^{v-2-j}(v-1-j) S_j( \psi_i^{{\G}_d}) }\\\label{D-value}
D(x,\G_d) &=&  \sum_{j=0}^{v-1}(vy)^{v-1-j}S_j( \psi_i^{\G_d}).
\end{eqnarray}

\begin{proposition}\label{proposition lower bound D} 
Let  $\delta=r(k-1)-\lambda(v-1)$ and suppose, $\Lam_2(v,b,k)\not=\Lam_3(v,b,k)$.
\begin{enumerate}

\item For given $v,b,k$ and $x_0=\frac{1}{v}\left((2\delta)^{v-1}\binom{v-1}{\lfloor\frac{v-1}{2}\rfloor}+1\right)$, if there exist designs with Laplacian matrix $L[x]$ with $L\in \Lam_3(v,b,k)$ for  $x\geq x_0$, then they are $D$-optimal among all designs with Laplacian matrix $L'[x]$ with $L'\in\Lam_2(v,b,k)$.

\item For given $v,b,k$ and $x_0=(2\delta)^{v-1}\binom{v-1}{\lfloor\frac{v-2}{2}\rfloor}^2(v-3)+\frac{1}{v}$, if there exist designs with Laplacian matrix $L[x]$ with $L\in \Lam_3(v,b,k)$ for  $x\geq x_0$, then they are $A$-optimal among all designs with Laplacian matrix $L'[x]$ with $L'\in\Lam_2(v,b,k)$.
\end{enumerate}

\end{proposition}
\begin{proof}
Let $L\in\Lam_2(v,b,k)$, $L = L[x,T_d]$ where $T_d=A_{\G_d}$  such that $S_3(\psi^{\G_d})\geq S_3(\rho^{\G'})$ for any $\delta$-regular simple graph $\G'$ on $v$ vertices.

\begin{enumerate}
\item
The Laplacian eigenvalues of a (not necessarily connected) $\delta$-regular graph are bounded from above by $2\delta$ \cite{Das}  and with Lemma \ref{D-value Bionom}  \[|S_j(\psi^{\G_d})- S_j(\psi^{\G'})|\leq \binom{v-1}{j}(2\delta)^j\leq(2\delta)^{v-1} \binom{v-1}{\lfloor\frac{v-1}{2}\rfloor},\ j=4,\ldots,v-1.\] It follows for $x\geq \frac{1}{v}\left((2\delta)^{v-1}\binom{v-1}{\lfloor\frac{v-1}{2}\rfloor}+1\right)$ that
\begin{eqnarray*}
 \sum_{j=4}^{v-1}v^{v-1-j}x^{v-1-j}|S_j(\psi^{\G_d})- S_j(\psi^{\G'})|
&\leq&(2\delta)^{v-1}\binom{v-1}{\lfloor\frac{v-1}{2}\rfloor}\frac{(vx)^{v-4}-1}{vx-1}\\
&\leq&(vx-1)\frac{(vx)^{v-4}-1}{vx-1}\\
&<&v^{v-4}x^{v-4}\\
&\leq& v^{v-4}x^{v-4}[S_3(\psi^{\G_d})-S_3(\psi^{\G'})].
\end{eqnarray*}
\item 
Similarly, with Lemma \ref{A-value Bionom} we have 
\[
|S_j(\psi^{\G_d}) S_{i-j}(\psi^{\G'})-S_j(\psi^{\G'}) S_{i-j}(\psi^{\G_d})|\leq (2\delta)^i\binom{v-1}{m}^2,\ j=4,\ldots,v-1.
\]
It follows for $x\geq (2\delta)^{v-1}\binom{v-1}{m}^2(v-3)+\frac{1}{v}$ that
\begin{eqnarray*}
 &&\sum_{i=4}^{v-1}(vx)^{2v-3-i}\sum_{j=0}^i(v-1-j)|S_j(\psi^{\G_d}) S_{i-j}(\psi^{\G'})-S_j(\psi^{\G'}) S_{i-j}(\psi_{\G_d})|\\
&\leq&(2\delta)^{v-1}\binom{v-1}{m}^2\sum_{i=4}^{v-1}(vx)^{2v-3-i}\sum_{j=0}^i(v-1-j)\\
&=&(2\delta)^{v-1}\binom{v-1}{m}^2\sum_{i=4}^{v-1}(vx)^{2v-3-i}(i+1)\frac{2(v-1)-i}{2}\\
&<&(2\delta)^{v-1}\binom{v-1}{m}^2v(v-3)(vx)^{v-2}\sum_{i=0}^{v-5}(vx)^{i}\\
&\leq&(vx-1)(vx)^{v-2}\frac{(vx)^{v-4}-1}{vx-1}\\
&<&(vx)^{2v-6}\\
&\leq& 3v^{2v-6}x^{2v-6}\left[S_3(\psi^{\G_d})-S_3(\psi^{\G'})\right].
\end{eqnarray*}
\end{enumerate}
\end{proof}
The bounds in Proposition \ref{proposition lower bound D} almost certainly can be improved; compuational results in our previous work showed $x_0 = \delta +1$ in all considered cases \cite{Cakiroglu}.

The RGD $d$ is fully characterized by the graph $\G_d$.
We therefore want to characterize the simple  regular graphs $\G$ on $v$ vertices that give rise to  $A$- and $D$-best RGDs for large $x$ in terms of the eigenvalues of $L(\G_d)$. We will need the following definition: the \emph{complement} $\bar{\G}$ of a graph $\G$ is the graph defined by the adjacency matrix $J_v - I_v - A_{\G}$, where $A_{\G}$ is the adjacency matrix of $\G$. 
The following proposition relates the $A$- and $D$-values given as in Equations (\ref{A-value as D/D'}) and  (\ref{D-value})  of with the complement $\bar{\G}$ of $\G$.

\begin{proposition}\label{proposition symmetric functions complement}
Let $\G$ be a simple (not necessarily connected) graph and $\bar{\G}$ its complement. Then
\[
S_j(\psi^{\bar{\G}})=\sum_{k=0}^j\binom{v-k-1}{j-k}(-1)^kv^{j-k}S_k(\psi^{\G}).
\] 
\end{proposition}
\begin{proof}
The non-trivial Laplacian eigenvalues of the complement $\bar{\G}$ are  $v-\psi_i^{\G}$ for $i=1,\ldots,v-1$. 

Therefore 
\begin{eqnarray*}
  S_j(\psi_{\bar{\G}})
 &=&\sum_{{\mathcal J}\subset {\mathcal I},|{\mathcal J}|=j}\prod_{i\in {\mathcal J}}(v-\psi_i^{\G})\\
&=&\sum_{k=0}^j(-1)^kv^{j-k}\sum_{\begin{subarray}{c} J\subseteq  {\mathcal I},\\|{\mathcal J}|=j\end{subarray}}\sum_{\begin{subarray}{c}  {\mathcal K}\subseteq {\mathcal J},\\|{\mathcal K}|=k\end{subarray}}\prod_{i\in {\mathcal K}} \psi_i^{\G}.
\end{eqnarray*}
In the last sum,  the elementary symmetric polynomial $S_k(\psi^{\G})$ occurs as often as the number of ways of choosing $j$ elements with a fixed  subset of  $k$ from a set of $v-1$ elements, that is $\binom{v-k-1}{j-k}$.
Hence,
\[
\sum_{\begin{subarray}{c} {\mathcal J}\subseteq  {\mathcal I},\\|{\mathcal J}|=j\end{subarray}}\sum_{\begin{subarray}{c} {\mathcal K}\subseteq {\mathcal J},\\|{\mathcal K}=k\end{subarray}}\prod_{i\in {\mathcal K}}\psi_i^{\G}=\binom{v-k-1}{j-k}S_k(\psi^{\G})
\]and the statement follows.
\end{proof}

\begin{lemma}\label{lemma S_3}
 Let  $\G$ and $\G'$ be simple  $\delta$-regular graphs on $v$ vertices. Then 
\[
S_3(\psi^{\G})-S_3(\psi^{\G'})=\frac{2}{3}(\eta(\G)-\eta(\G')),
\]where $\eta$ is the number of V-subgraphs of $\G$ and $\G'$, that is three vertices with exactly two edges.
\end{lemma}
\begin{proof}
  Since $\G$ and $\G'$ are $\delta$-regular,  $S_1(\psi^{\G})=S_1(\psi^{\G'})$ and $\sum_{j=1}^{v-1}(\psi_i^{\G})^2=\sum_{j=1}^{v-1}(\psi_i^{\G'})^2$. 
With $jS_j(z)=\sum_{l=1}^{j}(-1)^{l-1}S_{j-l}(z)\sum_{i=1}^{v-1}z_i^l$ for $z \in {\mathbb R}^{v-1}$ \cite{McDonald} it follows  that
\begin{eqnarray*}
2S_2(\psi^{\G})&=&S_1(\psi^{\G})^2-\sum_{j=1}^{v-1}(\psi_j^{\G})^2\\
3S_3(\psi^{\G})&=&\sum_{i=1}^3(-1)^{i-1}S_{3-i}(\psi^{\G})\sum_{j=1}^{v-1}(\psi_j^{\G})^j\\
&=&S_2(\psi^{\G})S_1(\psi^{\G})-S_1(\psi^{\G})\sum_{j=1}^{v-1}(\psi_j^{\G})^2+\sum_{j=1}^{v-1}(\psi_j^{\G})^3.
\end{eqnarray*}
and therefore  $S_2(\psi^{\G})=S_2(\psi^{\G'})$. The statement follows directly  with
\[
\sum_{i=1}^{v-1}(\psi_i^{\G})^2=v\delta(\delta+1)
\]and 
\[\sum_{i=1}^{v-1}(\psi_i^{\G})^3=v\delta(\delta+1)^2+2\eta(\G),\]where $\eta(\G)$ is the number of V-subgraphs of $\G$ (\cite{PetingiRodriguez} proved this equation for connected graphs but the argument directly extends to non-connected graphs). 
\end{proof}

\begin{theorem}\label{Corollary v-min D-best} 
Suppose, $\Lam_2(v,b,k)\not=\Lam_3(v,b,k)$. For given $v,b,k$, there exists an $x_0>0$ such that, if there exist designs with Laplacian matrix $L[x]$ with $L\in\Lam_2(v,b,k)$ for  $x\geq x_0$, then the ones whose underlying graph minimizes  the number of V-subgraphs in its complement are $A$- and $D$-optimal among all designs with Laplacian matrix $L'[x]$ with $L'\in\Lam_2(v,b,k)$.
\end{theorem}
\begin{proof}
This follows directly from Lemma \ref{lemma S_3} and the fact that  $\rho_i = v\lambda+\psi_i$ for $i = 1,\ldots,v-1$ and that $S_j(z)$ are increasing in $z$ for $j = 1,\ldots,v-1$.
\end{proof}

Let $K_{\alpha,\ldots,\alpha}$ denote the regular complete $m$-partite graph, that is a graph whose vertex set can be partitioned into $m$ groups of size $\alpha$ such that any pair of vertices is joined by an edge if and only if they are in different groups. An RGD with a multipartite concurrence graph and block size $2$ is a group divisible design. \cite{Cheng1} proved that regular complete bipartite graphs are the concurrence graphs of the unique $A$- and $D$-optimal designs for all $y\geq0$ (not necessarily only among RGDs) and extended his result to complete regular multipartite graphs for $y=0$. The following corollary extends  the latter result  to the BIBD-extension of the complete regular multipartite graph with a large number of blocks.

\begin{corollary}\label{corollary complete reg multipartite graph} 
There exists an $x_0>0$ such that if there exists a design with Laplacian matrix $L[x]$ where $L=L(K_{\alpha,\ldots,\alpha})\in\Lam_2(\alpha m,b,k)$ and $x\geq x_0$,  then it is $A$- and $D$-optimal among all designs with Laplacian matrix $L'[x]$ with $L'\in\Lam_2(\alpha m,b,k)$.
\end{corollary}
\begin{proof}
 Follows directly from  Theorem \ref{Corollary v-min D-best} and the fact that the complement is a union of cliques and as such V-subgraph-free.
\end{proof}

\newpage
\addcontentsline{toc}{section}{Bibliography}

\bibliographystyle{alpha}


\end{document}